\documentclass[12pt]{amsart}

\usepackage{psfrag}
\usepackage{color}
\usepackage{tikz}
\usetikzlibrary{matrix,arrows}
\usepackage{graphicx,graphics}
\usepackage{fullpage,amssymb,amsfonts,amsmath,amstext,amsthm,amscd,verbatim,enumerate}
\usepackage[T1]{fontenc}

\newtheorem{theorem}{Theorem}[section]
\newtheorem{lemma}[theorem]{Lemma}

\newtheorem{example}[theorem]{Example}
\numberwithin{equation}{section}

\theoremstyle{definition}
\newtheorem{definition}[theorem]{Definition}

\def\D{\Delta}

\def\to{\rightarrow }

\def\bar{\overline}

\def\C{\Bbb C }
\def\R{\Bbb R }

\def\N{\mathbb{N}}

\def\ft{\widetilde{f}}
\def\zbar{\overline{z}}

\begin{document}

\title{On locally quasiconformal Teichm\"uller spaces}

\author{Alastair Fletcher}

\address{ AF: Department of Mathematical Sciences, Northern Illinois University,
Dekalb, IL 60115, USA. E-mail address: fletcher@math.niu.edu}

\author{ Zhou Zemin}

\address{  ZZ: School of Mathematics, Renmin
University of China, Beijing, 100872, People's Republic of China.
 E-mail address: zzm@ruc.edu.cn }



\thanks{AF is supported by a grant from the Simons Foundation (\#352034, Alastair Fletcher). ZZ is partially
supported by the National Natural Science Foundation of China (Grant
 11571362).}

\date{}
\maketitle
\newtheorem{Theorem}{Theorem}

\begin{abstract}
We define a universal Teichm\"uller space for locally quasiconformal mappings whose dilatation grows not faster than a certain rate. Paralleling the classical Teichm\"uller theory, we prove results of existence and uniqueness for extremal mappings in the generalized Teichm\"uller class. Further, we analyze the circle maps that arise.
\end{abstract}

\section{Introduction}

Teichm\"uller theory is a major area of research in modern mathematics, bringing together analysis, geometry, topology and dynamics. The Teichm\"uller space of a topological surface parameterizes the set of complex structures that can be equipped on the surface. For example, not all tori are conformally equivalent and the space of complex structures of a genus one surface can be parameterized by the upper half-plane. A fundamental object in Teichm\"uller theory is the universal Teichm\"uller space of the disk, denoted $T(\D)$. Via the Uniformization Theorem, every Teichm\"uller space of a hyperbolic surface is embedded in $T(\D)$, and so it is an important object to understand. We refer to \cite{10,11,18,19} for introductions to Teichm\"uller theory.

There are various ways of modelling points of universal Teichm\"uller space. One can consider equivalence classes of quasiconformal maps $f:\D\to\D$ under the Teichm\"uller equivalence relation or, equivalently via solving the Beltrami differential equation $f_{\zbar} = \mu f_z$, equivalence classes of Beltrami differentials. We recall that Beltrami differentials are elements $\mu \in L^{\infty}(\D)$ with $||\mu ||_{\infty} <1$. Since every quasiconformal map $f:\D\to \D$ extends to a quasisymmetric homeomorphism $\ft :\partial \D \to \partial \D$, points of Teichm\"uller space can also be modelled as quasisymmetric maps of the circle which fix the three points $1,-1,i$.

The defining property of a quasiconformal map $f:\D\to \D$ is that it has uniformly bounded distortion. Every quasiconformal map has a complex dilatation $\mu_f = f_{\zbar} / f_z$ which is defined almost everywhere. The quasiconformality condition implies that there exists $0\leq k<1$ such that $||\mu_f||_{\infty} \leq k$ almost everywhere. Solving the Beltrami equation provides the converse to this statement. Recently, there has been interest in the consequences of allowing $||\mu||_{\infty}=1$ in the Beltrami equation and investigating properties of the solutions that occur.

In the literature, various classes of such mappings have been studied, for example David mappings \cite{9,Zakeri}, $\mu$-homeomorphisms \cite{4,5,6,7,8,13,14,17,24} and locally quasiconformal mappings \cite{21,25} have all been studied in dimension two. These all sit in the larger framework of mappings of finite distortion in Euclidean spaces, see for example \cite{12,HK,16}. Such mappings are far from novelties: Petersen and Zakeri used David mappings in \cite{PZ} to study Siegel disks in complex dynamics. Moreover, in the theory of length spectrum Teichm\"uller spaces, it is known that it certain circumstances it differs from the quasiconformal Teichm\"uller space, see for example the work of Shiga \cite{Shiga}. In particular, in this paper a certain map is constructed which has a uniform bound on the distortion of hyperbolic lengths of essential curves, but is (in our language) locally quasiconformal. It is therefore conceivable that locally quasiconformal mappings could play a role in the study of length spectrum Teichm\"uller spaces.

In the current paper, we will work in the setting of locally quasiconformal mappings of the disk and initiate the study of a universal Teichm\"uller space of such mappings. Our aim is to set up a workable definition, solve an extremality problem in this setting and study the boundary mappings that arise.

The paper is organized as follows: we recall some preliminary material on quasiconformal and locally quasiconformal mappings in section 2. In section 3, we define locally quasiconformal Teichm\"uller spaces and state our main results on them. In section 4, we provide proofs of our results. Finally in section 5, we give some concluding remarks indicating directions of future research.

{\bf Acknowledgements:} The second named author would like to thank Professor Chen Jixiu for his many useful suggestions and help.

\section{Preliminaries}

We recall some basic facts about quasiconformal mappings which can be found in many texts, for example \cite{1,10,11,16,18,19}.
Let $U\subset \C$ be a domain. The {\it distortion} of a homeomorphism $f:U\to V\subset \C$ is defined by
\[ D_f(z) = \frac{|f_z(z)| + |f_{\zbar}(z)|}{|f_z(z) -  |f_{\zbar}(z)|}.\]
A homeomorphism $f:U\to V$ is called $K$-quasiconformal if $f$ is absolutely continuous on almost every horizontal and vertical line in $U$ and moreover $\sup_{z\in U} D_f(z) \leq K$. The smallest such $K$ that holds here is called the {\it maximal dilatation} of $f$ and denoted $K_f$. If we do not need to specify the $K$, then we just call the map {\it quasiconformal}. The {\it complex dilatation} of $f$ is $\mu_f = f_{\zbar}/f_z$ and satisfies the equation
\[ D_f(z) = \frac{1+|\mu(z)|}{1-|\mu(z)|}.\]
If $f$ is quasiconformal, then there exists $0\leq k<1$ such that $||\mu_f ||_{\infty} \leq k$ and
\[ K_f = \frac{1+ ||\mu_f||_{\infty}}{1-||\mu_f ||_{\infty} }.\]
Every quasiconformal map $f:\D\to\D$ extends to a homeomorphism of $\overline{\D}$ and, moreover, the boundary map $\widetilde{f} :\partial \D\to \partial \D$ is quasisymmetric, that is, there exists $M\geq 1$ so that
\[ \frac{1}{M} \leq \frac{ | \widetilde{f}(e^{i(\theta +t)}) - \widetilde{f}(e^{i\theta}) |}{ | \widetilde{f}(e^{i\theta}) - \widetilde{f}(e^{i(\theta - t)}) |} \leq M \]
holds for all $\theta \in [0,2\pi)$ and $t>0$.

We now define the class of mappings that will form the basis of our study.

\begin{definition}
\label{def:lqc}
A homeomorphism $f:\D\to \D$ is called {\it locally quasiconformal} if and only if for every compact set $E\subset \D$, $f|_{E}$ is quasiconformal.
\end{definition}

This definition means we allow that distortion of our map to blow up as we head out towards the boundary. Moreover, the complex dilatation of a locally quasiconformal mapping is defined almost everywhere in the disk and we allow $||\mu_f||_{\infty}=1$. However, we have to restrict the types of locally quasiconformal mappings we study.

\begin{example}
\begin{enumerate}[(i)]
\item The map $f(z) = z(1-|z|^2)^{-1}$ from \cite{21} is a locally quasiconformal map from $\D$ onto $\C$. In this paper, we want to consider only locally quasiconformal self-mappings of $\D$.
\item The spiral map $f(re^{i\theta}) = r\exp(i(\theta + \ln \frac{1}{1-r} ) )$ is a locally quasiconformal map $f:\D\to \D$ but it does not extend continuously to the boundary. In this paper, we want to consider locally quasiconformal mappings which extend homeomorphically to the boundary.
\end{enumerate}
\end{example}

\begin{definition}
\label{def:rho}
We say that a continuous increasing function $\rho:[0,1) \to [1,\infty)$ is {\it allowable} if the following conditions hold:
\begin{enumerate}[(i)]
\item $\rho(0)=1$,
\item \hskip 1px \lefteqn{\int_0^1 \rho(r) \: dr < \infty,}
\item for some constant $R>0$ and every $\xi \in \partial \D$,
\[ \lim_{t\to 0^+} \int_t^R \frac{dr}{r\rho^*(\textcolor[rgb]{0.00,0.00,0.00}{|z|})} = +\infty,\]
where $\rho^*(r)$ is defined by
\begin{equation}
\label{eq:rhostar}\
\rho^*(r) = \int_{S(\xi ,r) \cap \D} \rho(|z|) d\theta,
\end{equation}
\textcolor[rgb]{0.00,0.00,0.00}{$z=\xi+r e^{i\theta}$}
and $S(\xi,r)$ is the circle centred at $\xi$ of radius $r$.
\end{enumerate}
\end{definition}

Each allowable $\rho$ yields a family of locally quasiconformal mappings.

\begin{definition}
Suppose $\rho$ is allowable. The family $QC_{\rho}(\D)$ consists of locally quasiconformal mappings $f:\D\to \D$ such that there exists $C>0$ with
\[ D_f(z) \leq C\rho(|z|),\]
for all $z\in \D$.
\end{definition}

Condition (i) in Definition \ref{def:rho} is a convenient normalization condition which implies that $D_{f}(z) / \rho (|z|) \leq 1$ for every conformal map $f:\D\to \D$ with equality at the origin. Condition (ii) implies by \cite[Theorem 1]{21} that if $\mu \in L^{\infty}(\D)$ with $|\mu(z)|<1$ then there exists a locally quasiconformal map $f:\D\to\D$ with complex dilatation $\mu$. Finally, condition (iii) implies by, for example, \cite[Theorem 1]{6} that $f$ extends continuously to $\partial \D$.

Observe that if $f$ is quasiconformal, then $f\in QC_{\rho}(\D)$ for every allowable $\rho$. Typically for a locally quasiconformal mapping, the maximal dilatation is not finite, but there is a maximal dilatation with respect to $\rho$.

\begin{definition}
Let $\rho$ be allowable and let $f\in QC_{\rho}(\D)$. Then the {\it maximal dilatation with respect to $\rho$} is defined by
\[ K^{\rho}_f:= \sup_{z\in \D} \frac{ D_f(z)}{\rho(|z|)}.\]
\end{definition}

As remarked above, condition (i) in Definition \ref{def:rho} implies that $K^{\rho}_{f}=1$ for every conformal map $f$ and every allowable $\rho$.

\begin{example}
By \cite[Theorem 4]{21}, if $\rho(r) = \log \frac{1}{1-r}$, then any locally quasiconformal map $f:\D \to \D$ with $D_f(z) \leq C \rho (|z|)$ extends homeomorphically to $\overline{\D}$. It is a short computation to show that, if $\sigma$ denotes Lebesgue measure, then any such $f$ satisfies
\[ \sigma \{ z\in \D : D_f(z) >K \} < \pi e^{-2K/C}.\]
In particular, this means that any such $f$ is a David mapping. We don't know how $QC_{\rho}(\D)$ and David mappings are related in general.
\end{example}

Since $\rho(0)=1$, we have $K^{\rho}_f \geq 1$, and so $K^{\rho}_f$ gives a quantity the describes how far $f$ is from a conformal map, except that here $K^{\rho}_f = 1$ does not imply that $f$ is conformal. To see this, we only need $D_f(z)$ to grow slower than $\rho(|z|)$.

\section{Locally quasiconformal Teichm\"uller spaces}

In this section we define locally quasiconformal Teichm\"uller spaces with respect to an allowable $\rho$ and state our results.

\begin{definition}
\label{def:lqct}
Let $\rho$ be allowable. Then the set $\mathcal{L}_{\rho}(\D)$ consists of locally quasiconformal mappings $f:\D\to\D$ such that $f^{-1} \in QC_{\rho}(\D)$.
\end{definition}

We need to control the growth of $f^{-1}$ for our results. Note that such a condition on $f$ does not imply the same condition holds for $f^{-1}$, in contrast to the fact that the inverse of a $K$-quasiconformal map is also a $K$-quasiconformal map.

\begin{example}
\label{ex:radial}
Consider radial maps given in polar coodinates by $f_a(re^{i\theta}) = [1-(1-r)^a]e^{i\theta}$ for $a>0$.
A computation shows that
\[ D_f(re^{i\theta}) = \frac{1-(1-r)^a}{ar(1-r)^{a-1}},\]
and the right hand side is the appropriate $\rho$ to consider for such a map. However, condition (ii) in Definition \ref{def:rho} is only satisfied when $0<a<2$. Since $f_a^{-1} = f_{1/a}$ and $f_a\circ f_b = f_{ab}$, we see that condition (ii) is not closed under inverses and compositions.
\end{example}

\begin{lemma}
\label{lem:1}
If $f\in \mathcal{L}_{\rho}(\D)$, then $f$ can be extended homeomorphically to a map $\overline{\D}\to \overline{\D}$.
\end{lemma}

\begin{proof}
Since $\rho$ is allowable and $f\in\mathcal{L}_{\rho}(\D)$, then $D_{f^{-1}}(z) \leq C \rho(|z|)$, where $\rho$ satisfies the conditions in Definition \ref{def:rho}. It immediately follows from \cite[Theorem 1.1]{25} that $\mu_{f^{-1}}$
can be integrated to give a locally quasiconformal map $g$ which extends to a homeomorphism of $\overline{\D}$.
Now, for $z\in \D$, the complex dilatation of $g\circ f$ satisfies
\[ \mu_{g\circ f}(f^{-1}(z)) = \omega (z) \cdot \frac{ \mu_g(z) - \mu_{f^{-1}}(z) }{1-\overline{\mu_{f^{-1}}(z) }\mu_g(z) } \equiv 0,\]
since $\mu_g = \mu_{f^{-1}}$, and where $|\omega(z)|=1$ for all $z\in \Delta$. Hence $g\circ f$ is a conformal map from $\D$ to itself which extends to the boundary. Hence $f$, and also $f^{-1}$, extend homeomorphically to $\overline{ \D}$.
\end{proof}

Recall that $A(\D)$ is the Bergman space of integrable holomorphic functions on $\D$, that is,
\[ A(\D) = \{ \varphi : \int_{\D} |\varphi| < \infty \} .\]
Complex dilatations $\mu(z) = k\overline{\varphi}/|\varphi|$ for $0<k<1$ and $\varphi \in A(\D)$ are said to be of Teichm\"uller-type
and play an important role in extremal problems in Teichm\"uller theory. We show that an analogue exists for locally quasiconformal mappings.

\begin{lemma}
\label{lem:2}
Let $\varphi_0 \in A(\Delta)$ be not identically zero,
$K_0\geq 1$ be a constant and let $\rho$ be allowable. Then there exists a locally quasiconformal mapping $f_0\in \mathcal{L}_{\rho}(\D)$ with Teichm\"uller-type complex dilatation
\begin{equation}
\label{eq:lem2}
\mu_{f_0}(z) = \frac{\rho(|f_0(z)|)K_0-1}{\rho(|f_0(z)|)K_0+1}\cdot \frac{\overline{\varphi_0 (z)}}{|\varphi_0 (z)|}.
\end{equation}
\end{lemma}

We prove this lemma in the next section. Since elements of $\mathcal{L}_{\rho}(\D)$ extend homeomorphically to the boundary, we may always assume that we have post-composed by a M\"obius map so that elements of $\mathcal{L}_{\rho}(\D)$ fix $1,-1$ and $i$.

\begin{definition}
\label{def:teich}
Let $\rho$ be allowable. Then $f,g\in \mathcal{L}_{\rho}(\D)$ are Teichm\"uller related with respect to $\rho$, denoted by $f\sim g$, if and only if the boundary extensions of $f$ and $g$, normalized to fix $1,-1,i$, agree. We then define the {\it generalized Teichm\"uller space with respect to $\rho$} by
\[ T_{\rho}(\D) = \mathcal{L}_{\rho}(\D) / \sim.\]
\end{definition}

Elements of $T_{\rho}(\D)$ are Teichm\"uller equivalence classes denoted by $[f]_{\rho}$, or simply $[f]$ if the context is clear.
Given $[f_0] \in T_{\rho}(\D)$, every representative of $[f_0]$ is a locally quasiconformal map $f$ whose boundary values agree with $f_0$ and so that $K^{\rho}_{f^{-1}} < \infty$.

\begin{definition}
\label{def:extremal}
Let $[f_0] \in T_{\rho}(\D)$.
\begin{enumerate}[(i)]
\item We say that $f\in [f_0]$ is {\it extremal} if $K^{\rho}_{f^{-1}} \leq K^{\rho}_{g^{-1}}$ for all $g\in [f_0]$.
\item We say that $f\in [f_0]$ is {\it uniquely extremal} if $K^{\rho}_{f^{-1}} < K^{\rho}_{g^{-1}}$ for all $g\in [f_0] \setminus \{ f \}$.
\end{enumerate}
\end{definition}

Our first result on extremal maps is that extremal representatives always exist.

\begin{theorem}
\label{thm:extremal}
Let $\rho$ be allowable and let $[f_0] \in T_{\rho}(\D)$. Then there exists an extremal representative $f\in [f_0]$.
\end{theorem}

We show next that uniquely extremal representatives exist.

\begin{theorem}
\label{thm:ue}
Let $\rho$ be allowable, $K_0>1$, $\varphi_0 \in A(\D)$ and suppose $f_0 \in T_{\rho}(\D)$ has Teichm\"uller-type complex dilatation
\[ \mu_{f_0}(z) = \frac{ \rho(|f_0(z)|)K_0 - 1}{\rho(|f_0(z)|)K_0+1} \cdot \frac{ \overline{\varphi_0(z)}}{|\varphi_0(z)|}.\]
Then $f_0$ is uniquely extremal in $[f_0]$.
\end{theorem}

We next turn to the boundary mapping induced by an element $[f] \in T_{\rho}(\D)$. It is well-known that every quasiconformal self-map of $\D$ extends to a quasisymmetric map of the unit circle and, conversely, every quasisymmetric map of the unit circle extends to a quasiconformal map of $\D$. For $h:\partial \D \to \partial \D$, we define the quasisymmetric function
\[ \lambda_h(\xi, t) =  \frac{ | h(\xi e^{it}) - h(\xi) |}{ | h(\xi) - h(\xi e^{-it}) |} .\]
This is the circle version of the standard quasisymmetric function considered by many authors, for example \cite{6,Zakeri}.

\begin{theorem}
\label{thm:qs}
Let $\rho$ be allowable and $[f]\in T_{\rho}(\D)$ with associated circle map $h$. Then there exists a function $\lambda$ depending only on $\rho$ such that
\[ \frac{1}{\lambda(t)} \leq \lambda_{h^{-1}}(\theta, t) \leq \lambda(t)\]
for all $\theta \in [0,2\pi)$.
\end{theorem}

In proving this result we will give an explicit formula for $\lambda(t)$. Note that we may have $\lambda(t) \to \infty$ as $t\to 0$.
Theorem \ref{thm:qs} is a circle version of a result proved for the extended real line in \cite{6}. However, the extended real line has a special point at infinity, whereas the circle has no special point. In particular, it is not a trivial task to obtain our result from \cite{6}.

\section{Proofs of results}

\subsection{Generalized Teichm\"uller-type dilatations}
Here we prove Lemma \ref{lem:2}. First, we recall some results on non-linear elliptic systems from \cite[Chapter 8]{3}.
Assume that $h:\C\times\C\times\C \to \C$ satisfies the following conditions:
\begin{enumerate}[(i)]
\item the homogeneity condition that $ f_{\bar{z}} = 0$ whenever $f_z = 0$, or equivalently,
\[ H(z,w,0)\equiv 0,\,\, \text{for almost every}\,\, (z,w)\in \C\times \C; \]
\item the uniform ellipticity condition that for almost every $z,w \in \C $ and all $\zeta,\xi \in \C,$
\[ |H(z,w,\zeta)-H(z,w,\xi)|\leq k|\zeta -\xi|,\]
for some $0\leq k<1$;
\item $H$ is Lusin measurable (see \cite[p.238]{3} for further details).
\end{enumerate}
A solution $f\in W^{1,2}_{loc}(\C)$ to
\begin{equation}
\label{eq:elliptic}
\frac{\partial f}{ \partial \bar{ z}}= H(z,\,f,\,\frac{\partial f}{\partial z}),
\end{equation}
for $z\in \C$ and normalized by the condition
\[ f(z)=z+a_1z^{-1}+a_2z^{-2}+\cdots\]
outside a compact set will be called a {\it principal solution}. A homeomorphic solution $f\in W^{1,2}_{loc}(\C)$ to \eqref{eq:elliptic} is called {\it normalized} if $f(0)=0$ and $f(1)=1$. Naturally any such solution fixes the point at infinity too.

\begin{lemma}[Theorem 8.2.1, \cite{3}]
\label{lem:h}
Under the hypotheses above, equation \eqref{eq:elliptic} admits a normalized solution. If, in addition, $H(z,w,\zeta)$ is compactly supported in the $z$-variable, then the equation admits a principal solution.
\end{lemma}

We now prove Lemma \ref{lem:2}.

\begin{proof}[Proof of Lemma \ref{lem:2}]
Recalling \eqref{eq:lem2}, the equation we want to solve is
\[ \frac{\partial f_0}{\partial \bar{z}}=\frac{\rho(|f_0(z)|)K_0-1}{\rho(|f_0(z)|)K_0+1}\cdot \frac{\overline{\varphi_0}}{|\varphi_0|}\cdot \frac{\partial f_0}{\partial z}.\]
To that end, let
\[ H(z,w,\zeta)=H(z,f_0,\frac{\partial f_0}{\partial z})=\frac{\rho(|f_0|)K_0-1}{\rho(|f_0|)K_0+1}\cdot \frac{\overline{\varphi_0(z)}}{|\varphi_0(z)|}\cdot\frac{\partial f_0}{\partial z}.\]
For $n= 2,3,\ldots$, set $\D_n = \{z:\,|z|<1-\frac{1}{n}\}$ and $D_n=\{z:\,1-\frac{1}{n}< |z|<1-\frac{1}{n^2}\}$.
By interpolating in $\overline{D_n}$ with a continuous function $F_n$ satisfying $|F_n(z,w,\zeta)| \leq |H(z,w,\zeta)|$, there exists a continuous function
\[ H_n(z,w,\zeta):=\begin{cases}H(z,w,\zeta)\,\,\quad&\text{when}\,\,z\in \Delta_n  \\ F_n(z,w,\zeta)\,\,\quad&\text{when}\,\,z\in \overline{ D_n} \\
0&\text{when}\,\,z\in
\C-\overline{\Delta_n}-\overline{D_n}\end{cases}.\]
Applying Lemma \ref{lem:h}, we see that the equation
\[ \frac{\partial f}{\partial \bar{z}}= H_n(z,w,\frac{\partial f}{\partial z})\]
admits a principal and normalized solution $f_n:\C\to\C$.

For every fixed $j$, the family $\{f_n|_{\Delta_j} \}$ consists of quasiconformal mappings with a uniform bound on the distortion. Since $f_n(0) = 0$ and $f_n(1)=1$ for all $n$, the family is uniformly bounded in $\Delta_j$. It follows from the quasiconformal version of Montel's Theorem (see \cite{Miniowitz}) that the family $\{f_n|_{\Delta_j} \}$ is normal. By a standard diagonal argument, we can find a subsequence $(f_{n_p})_{p=1}^{\infty}$ which converges locally uniformly in $\D$. Suppose $f_0$ is the limit function. The image $f_0(\D)$ may not be $\D$, but it must be a simply connected proper subset of $\C$. By post-composing by a suitable conformal map, via the Riemann Mapping Theorem, we can assume that $f_0:\D\to \D$ and $f_0$ still fixes $0$ and $1$. Then $f_0$ is either a locally quasiconformal mapping (since it is quasiconformal on each $\D_j$) with complex dilatation \eqref{eq:lem2} or a constant. However, since $f_0(0) \neq f_0(1)$, $f_0$ cannot be constant.

Finally, we need to show that $f_0 \in \mathcal{L}_{\rho}(\D)$. Set $g=f_0^{-1}$. Then, by the formula for the complex dilatation of an inverse (see \cite[p.6]{10}), for $z\in \D$ we have
\[ |\mu_g(z)| = |\mu_{f_0}(f_0^{-1}(z))| = \left | \frac{ \rho(|z|){K_0}-1}{\rho(|z|)K_0+1} \cdot \frac{ \overline{\varphi (f_0^{-1}(z))}}{|\varphi(f_0^{-1}(z)) |} \right | = \frac{\rho(|z|)K_0-1}{\rho(|z|)K_0+1}.\]
We therefore obtain
\[  \frac{1+|\mu_g(z)|}{1-|\mu_g(z)|} \leq K_0\rho(|z|)\]
and conclude that $g\in QC_{\rho}(\D)$ and hence $f_0 \in \mathcal{L}_{\rho}(\D)$. The proof is complete.
\end{proof}

\subsection{Extremal Mappings}

We will show that extremal mappings always exist, but first we need to prove a normal family result which generalizes \cite[Theorem 3.1]{7}.

\begin{lemma}
\label{lem:normal}
Let $\rho$ be allowable, suppose $\mathcal{F} \subset QC_{\rho}(\D)$ and there exists a constant $C>0$ so that
\[ D_f(z) \leq C\rho(|z|)\]
for all $f\in \mathcal{F}$. Then $\mathcal{F}$ is a normal family and relatively compact in $QC_{\rho}(\D)$ viewed as a subset of continuous functions from the disk to itself.
\end{lemma}

\begin{proof}
Let $r<1$. Then on $\D_r =\{z:|z|<r \}$, every $f\in \mathcal{F}$ is $C\rho(r)$-quasiconformal with image contained in $\Delta$. By Montel's Theorem for quasiconformal mappings (see \cite{Miniowitz}), if $(f_n)_{n=1}^{\infty}$ is any sequence in $\mathcal{F}$, then $(f_n|_{\Delta_r})_{n=1}^{\infty}$ contains a subsequence $(f_{n_k}|_{\Delta_r})_{k=1}^{\infty}$ which converges to either a $C\rho(r)$-quasiconformal map or a constant.

By a standard diagonal sequence argument, we find a subsequence $(f_{n_p})_{p=1}^{\infty}$ with $f_{n_p}$ converging to $f_0$ locally uniformly on $\D$. Then $f_0$ is either a constant or a locally quasiconformal map with $D_{f_0}(z) \leq C\rho(|z|)$. Hence $f_0 \in QC_{\rho}(\Delta)$.
\end{proof}

\begin{proof}[Proof of Theorem \ref{thm:extremal}]
Let $[f_0]\in T_{\rho}(\D)$. Then for each $g\in [f_0]$ we have $1\leq K^{\rho}_{g^{-1}} <\infty$. Let
\[ K = \inf_{g\in [f_0]} K^{\rho}_{g^{-1}},\]
and find $f_n \in [f_0]$ with $K^{\rho}_{f_n^{-1}} \to K$. Without loss of generality, we may assume that $K^{\rho}_{f_n^{-1}}$ is decreasing. Set $C = K^{\rho}_{f_1^{-1}}$. Then for all $n\in \N$, we have
\[ D_{f_n^{-1}}(z) \leq C \rho(|z|).\]
By Lemma \ref{lem:normal}, the family $\{ f_n^{-1} : n\in \N \}$ is normal and hence there exists a subsequence $(f_{n_k}^{-1})_{k=1}^{\infty}$ which converges locally uniformly on $\D$ to a continuous map $h$. For each $k$, $f_{n_k}^{-1}$ agrees with $f_0^{-1}$ on $\partial \D$ and hence $h$ cannot be a constant. We conclude that $h\in QC_{\rho}(\D)$ and $K^{\rho}_h = K$. Setting $f=h^{-1}$ we see that $f$ is extremal in $[f_0]$.
\end{proof}

\subsection{Uniquely Extremal Mappings}

\begin{proof}[Proof of Theorem \ref{thm:ue}]
By Lemma \ref{lem:2}, there exists a locally quasiconformal map $f_0$ with complex dilatation of Teichm\"uller type given by \eqref{eq:lem2}. Consequently, the Teichm\"uller class $[f_0]$ is well-defined and at least contains the representative $f_0$.

Let $f\in [f_0]$ and set $g=f^{-1}\circ f_0$. As we observed in Example \ref{ex:radial}, elements of $QC_{\rho}(\D)$ are not necessarily preserved by taking inverses or compositions. However, $g$ is locally quasiconformal in $\D$, extends to a homeomorphism from $\overline{\D}$ to itself, has generalized partial derivatives on $\D$ and is the identity on $\partial \D$. We can therefore apply the generalized version of the Reich-Strebel Main Inequality proved by Markovi\'c and Mateljevi\'c \cite[Theorem 1]{22} to see that for any $\varphi \in A(\D)$, we have
\begin{equation}
\label{eq:ue1}
\int_{\D} |\varphi| \leq \int_{\D} |\varphi| \frac{ |1+ \mu_g \varphi  |\varphi| |^2}{1-|\mu_g|^2}.
\end{equation}
For convenience, we write $\mu_1(z) = \mu_{f^{-1}}(f_0(z))$, $\mu = \mu_{f_0}$ and $\tau = \overline{(f_0)_z} / (f_0)_z$.
Since
\[ \mu_g=\frac{\mu+\mu_1\tau}{1+\bar{\mu}\mu_1\tau}\]
we obtain
\begin{equation}
\label{align:ue}
 \frac{|1+\mu_g\varphi/|\varphi||^2}{1-|\mu_g|^2} \leq  \frac{|1+\mu\varphi/|\varphi||^2}{(1-|\mu|^2)}
\frac{|1+\mu_1\tau\frac{\varphi}{|\varphi|}(1+\overline{\mu}\frac{\overline{\varphi}}{|\varphi|})/(1+\mu\frac{\varphi}{|\varphi|})|^2}{(1-|\mu_1|^2)}.
\end{equation}
Consequently
\begin{equation}
\label{eq:ue2}
\int_{\D}|\varphi| \leq \int_{\D} |\varphi| \frac{|1+\mu\varphi/|\varphi||^2}{(1-|\mu|^2)}
\frac{|1+\mu_1\tau\frac{\varphi}{|\varphi|}(1+\overline{\mu}\frac{\overline{\varphi}}{|\varphi|})/(1+\mu\frac{\varphi}{|\varphi|})|^2}{(1-|\mu_1|^2)}.
\end{equation}
Setting $\varphi = -\varphi_0$ in \eqref{eq:ue2} and using \eqref{eq:lem2}, we obtain
\begin{align*}
\int_{\D} |\varphi_0| & \leq \int_{\D} |\varphi_0| \frac{1}{K_0\rho( |f_0(z) |)} \frac{|1+\mu_1\tau|^2}{1-|\mu_1|^2} \\
&\leq \frac{1}{K_0} \int_{\D} |\varphi_0| \frac{1}{\rho( |f_0(z)| ) } \frac{1+|\mu_1|}{1-|\mu_1|} \\
&\leq \frac{1}{K_0} \int_{\D} |\varphi_0| \left | \left | \frac{ 1+|\mu_1|}{1-|\mu_1|} \frac{1}{\rho(|f_0(z)|)}\right | \right |_{\infty} \\
&= \frac{K^{\rho}_{f^{-1}} }{K_0} \int_{\D} |\varphi_0|.
\end{align*}
We conclude that $K_0 \leq K^{\rho}_{f^{-1}}$. Since $K^{\rho}_{f_0^{-1}} = K_0$, we see that $f_0$ is extremal.

To show that $f_0$ is uniquely extremal, suppose $f\in [f_0]$ is extremal and keep the same notation as above. Therefore $K^{\rho}_{f^{-1}} = K^{\rho}_{f_0^{-1}}$. We obtain from \eqref{eq:lem2}, \eqref{eq:ue2} and setting $\varphi = -\varphi_0$ that
\[ \int_{\D} |\varphi| \leq \int_{\D} |\varphi| \frac{1}{\rho(|f_0(z)|)K_0} \frac{ |1-\mu_1 \tau \varphi_0/|\varphi_0||^2}{1-|\mu_1|^2}.\]
Since
\[   \frac{ |1-\mu_1 \tau \varphi_0/|\varphi_0||^2}{1-|\mu_1|^2} \leq \frac{ (1+|\mu_{f^{-1}}(f_0(z))|)^2}{1-|\mu_{f^{-1}}(f_0(z))|^2} \leq \rho(|f_0(z)|) K^{\rho}_{f^{-1}},\]
it follows that
\[ \int_{\D} |\varphi| \leq \int_{\D} |\varphi| \frac{1}{\rho(|f_0(z)|)K_0} \rho(|f_0(z)|) K^{\rho}_{f^{-1}} = \int_{\D} |\varphi|.\]
Since there must be equality everywhere, in particular we must have
\[ \mu_1(z) = \mu_{f^{-1}}(f_0(z))= -\frac{1}{\tau} \frac{ K_0\rho(|f_0(z)|)-1}{K_0\rho(|f_0(z)|)+1} \frac{\overline{\varphi_0}}{|\varphi_0|}.\]
However, we also have from the definition of $f_0$ that
\[ \mu_{f_0^{-1}}(f_0(z)) = -\frac{1}{\tau} \mu_{f_0}(z) =
 -\frac{1}{\tau} \frac{ K_0\rho(|f_0(z)|)-1}{K_0\rho(|f_0(z)|)+1} \frac{\overline{\varphi_0}}{|\varphi_0|}.\]
Since $f_0^{-1}$ and $f^{-1}$ have the same complex dilatation, then the same argument as in the proof of Lemma \ref{lem:1} shows that they are related via post-composition by a conformal map. However, since both $f_0^{-1}$ and $f^{-1}$ agree on $\partial \D$, the conformal map must be the identity. We conclude that $f= f_0$ and so $f_0$ is uniquely extremal.
\end{proof}

\subsection{Quasisymmetry functions}

Before we prove Theorem \ref{thm:qs}, we need to recall some results. If $\Gamma$ is a curve family, then let $M(\Gamma)$ denote its modulus. We refer to \cite[Chapter II]{Vuorinen} for the precise definition. In particular, if $E,F \subset \C$ are disjoint continua, then $\Delta(E,F)$ denotes the family of curves starting in $E$ and terminating in $F$ and $M(\Delta(E,F))$ is the corresponding modulus.

For $\xi \in \partial \D$ and $0<r<R$, let $Q(\xi,r,R)$ be the quadrilateral
\[ Q(\xi,r,R) = \{ z: r\leq |z-\xi | \leq R, z\in \D \} \]
with vertices taken in order as the intersections of the circle $\{ |z-\xi| = r \}$ and $\{ |z - \xi| = R \}$ with the unit circle. The modulus $\operatorname{mod} Q(\xi,r,R)$ is then defined as the modulus of the curve family joining the two components of $Q(\xi,r,R) \cap \partial \D$. If $f:\overline{\D}\to \overline{\D}$ is a homeomorphism, then we define $\operatorname{mod} f(Q(\xi,r,R))$ analogously.

\begin{lemma}[Lemma 2.1, \cite{25}]
\label{lem:qs}
Let $\rho$ be allowable and let $f\in QC_{\rho}(\D)$. Then
\[ \int_r^R \frac{dt}{t\rho^*(t)} \leq \operatorname{mod} f( Q(\xi,r,R)),\]
where \[\rho^*(t) = \int_{S(\xi ,t) \cap \D} \rho(|z|) d\theta,\]
and $z=\xi+t e^{i\theta}$.
\end{lemma}

Let $\tau:(0,\infty) \to (0,\infty)$ be the Teichm\"uller capacity function (see \cite[p.66]{Vuorinen}). Observe that $\tau$ is decreasing.

\begin{lemma}[Lemma 7.34, \cite{Vuorinen}]
\label{lem:v}
If $\Omega \subset \C$ is an open ring with complementary components $E,F$ and $a,b\in E$, $c,\infty \in F$, then
\[ M(\Delta(E,F)) \geq \tau \left ( \frac{ |a-c|}{|a-b|} \right ).\]
\end{lemma}

\begin{proof}[Proof of Theorem \ref{thm:qs}]
Since $[f] \in T_{\rho}(\D)$, if $g=f^{-1}$, then $g \in QC_{\rho}(\D)$ and extends to a boundary map that, by abuse of notation, we will also call $g$.

Let $\xi \in \partial \D$, $t>0$ and $w = \xi e^{it/2} \in \partial \D$ be the midpoint of the arc of $\partial \D$ from $\xi$ to $\xi e^{it}$. Let $Q$ be the quadrilateral $Q=Q(w,s,S)$, where $s=|\xi-w| = |e^{it/2}-1|$ and $S = |\xi e^{-it} - w| = |e^{3it/2}-1|$. Then since $g\in QC_{\rho}(\D)$, by Lemma \ref{lem:qs},
\begin{equation}
\label{eq:qs1}
\int_s^S \frac{dr}{r \rho^*(r) } \leq \operatorname{mod} f(Q).
\end{equation}
Let $I(z)  = z/|z|^2$ be inversion in the unit circle. Then $\Omega:= f(Q) \cup I(f(Q))$ is a ring domain. If $\Gamma$ is the curve family separating boundary components of $\Omega$, by symmetry we have
\begin{equation}
\label{eq:qs2}
M(\Gamma) = \frac{ \operatorname{mod}{f(Q)}}{2}.
\end{equation}
Now, if $\Gamma'$ is the curve family connecting the complementary components of $\Omega$, then
\begin{equation}
\label{eq:qs3}
M(\Gamma ') = 1/ M(\Gamma).
\end{equation}
Applying Lemma \ref{lem:v} with $a=g(\xi)$, $b = g(\xi e^{it})$ and $c = g(\xi e^{-it})$, we have
\begin{equation}
\label{eq:qs4}
M(\Gamma ') \geq \tau \left ( \frac{ | g(\xi) - g(\xi e^{-it}) |}{| g(\xi) - g(\xi e^{it})| } \right)
= \tau \left ( \frac{1}{\lambda_g(\xi, t) }\right ).
\end{equation}
Combining \eqref{eq:qs1}, \eqref{eq:qs2}, \eqref{eq:qs3} and \eqref{eq:qs4}, we conclude that
\[ \int_s^S \frac{dr}{r \rho^*(r) } \leq 2\left ( \tau \left ( \frac{1}{\lambda_g(\xi, t) }\right ) \right )^{-1}.\]
Rearranging in terms of $\lambda_g(\xi, t)$ and using the fact that $\tau$ is decreasing, we obtain
\[ \lambda_g(\xi, t)  \leq \left [ \tau^{-1} \left ( \frac{2}{\int_s^S \frac{dr}{r \rho^*(r) }} \right ) \right ] ^{-1}.\]

For the reverse inequality, we apply the same argument as above, except this time we let $w' = \xi e^{-it/2}$ be the midpoint of the arc of $\partial \D$ between $\xi$ and $\xi e^{-it/2}$ and we let $Q'$ be the quadrilateral $Q(w',s,S)$. Using the same notation as above, the argument is the same until we reach \eqref{eq:qs4} and we obtain
\[ M(\Gamma ') \geq \tau \left ( \frac{ | g(\xi) - g(\xi e^{it}) |}{| g(\xi) - g(\xi e^{-it})| } \right)  = \tau \left ( \lambda_g(\xi, t) \right ).\]
This yields
\[ \lambda_g(\xi ,t) \geq  \tau^{-1} \left ( \frac{2}{\int_s^S \frac{dr}{r \rho^*(r) }} \right ).\]
Consequently, we obtain the desired quasisymmetry estimate with
\[ \lambda (t) = \left [ \tau^{-1} \left ( \frac{2}{\int_s^S \frac{dr}{r \rho^*(r) }} \right ) \right ] ^{-1},\]
recalling that $s = |e^{it/2}-1|$ and $S = |e^{3it/2}-1|$.
\end{proof}

\section{Concluding remarks}

\subsection{Boundary maps}

In Theorem \ref{thm:qs}, we showed that if $[f] \in T_{\rho}(\D)$ and $f_0$ is any representative, then $f_0^{-1}$ extends to the boundary and the boundary map has controlled quasisymmetry function $\lambda$ depending on $\rho$ which may, however, blow up. What is not clear is whether given a boundary map with quasisymmetry controlled by $\lambda$, there is a locally quasiconformal extension contained in $QC_{\rho}(\D)$. If so, then we would have an alternate parameterization of $T_{\rho}(\D)$ through boundary maps of controlled quasisymmetry, in analogy with the quasisymmetric parameterization of universal Teichm\"uller space.

There are various extensions available. The Douady-Earle extension \cite{DE} extends a homeomorphism of the circle to a diffeomorphism of the (open) disk and hence this extension will be locally quasiconformal. It would be interesting to know how the distortion of the extension is controlled by $\lambda$.

Another extension is obtained through the Beurling-Ahlfors extension of homeomorphisms of the real line to homeomorphisms of the upper half-plane. In \cite{6}, it is shown that control of the quasisymmetry function of the boundary map leads to control of the distortion of the Beurling-Ahlfors extension, but again we don't know whether we can obtain $\rho$ from $\lambda$.

\subsection{Pseudo-metrics and metrics}

It is well-known that universal Teichm\"uller space carries the Teichm\"uller metric, and this has been well studied. For $T_{\rho}(\D)$, it is not clear how to make it into a metric space. One immediate barrier is that different classes can have extremal representatives $f_1,f_2$ which both have $K^{\rho}_{f_i^{-1}} = 1$ for $i=1,2$. Moreover, the fact that $QC_{\rho}(\D)$ need not be closed under compositions and inverses makes a direct analogy of the Teichm\"uller metric impossible.

On the other hand, it is easy to turn $T_{\rho}(\D)$ into a pseudo-metric space by considering maps $F:T_{\rho}(\D) \to (X,d_X)$, where $X$ is a metric space with distance function $d_X$. We then define $d_{T,X}$ on $T_{\rho}(\D)$ via
\[ d_{T,X}( [f] , [g] ) = d_X( F([f]) , F([g]) ).\]
For example $F([f]) = \inf_{f_0\in [f]} K^{\rho}_{f_0}$ mapping $T_{\rho}(\D)$ into $\R^+$ with the usual Euclidean metric yields a pseudo-metric. We can therefore ask how to construct a metric on $T_{\rho}(\D)$ or, slightly nebulously, how to construct as interesting a pseudo-metric as possible.

\subsection{Riemann surfaces}

Our construction of generalized universal Teichm\"uller spaces leads to an obvious generalization to generalized Teichm\"uller spaces of hyperbolic Riemann surfaces, that is, those surfaces covered by the unit disk. It seems plausible that a study of such objects could yield information about the interplay between length-spectrum and quasiconformal Teichm\"uller spaces. In particular, does length spectrum Teichm\"uller space sit inside a generalized Teichm\"uller space for every infinite-type Riemann surface?

\end{document}